\documentclass[a4paper,reqno]{amsart}

\usepackage[utf8]{inputenc}
\usepackage{microtype}

\usepackage{amssymb}
\usepackage{mathrsfs}
\usepackage{mathtools}
\usepackage{tikz}
\usepackage{hyperref}

\newtheorem{theorem}{Theorem}[section]
\newtheorem*{maintheorem}{Main Theorem}
\newtheorem{lemma}[theorem]{Lemma}
\newtheorem{proposition}[theorem]{Proposition}
\newtheorem*{brownscriterion}{Brown's Criterion}
\newtheorem{corollary}[theorem]{Corollary}

\theoremstyle{definition}
\newtheorem{definition}[theorem]{Definition}

\newtheorem{observation}[theorem]{Observation}

\numberwithin{equation}{section}

\newcommand{\serialcomma}{, }

\newcommand{\Z}{\mathbb{Z}}
\newcommand{\N}{\mathbb{N}}

\newcommand{\M}{\mathcal{M}}

\DeclareMathOperator{\Stab}{Stab}
\DeclareMathOperator{\lk}{lk}

\DeclareMathOperator{\dlk}{{\lk}{\downarrow}}

\newcommand{\NDInterval}{I}

\renewcommand{\:}{\colon}

\newcommand{\Poset}{\mathcal{P}}
\newcommand{\prePoset}{\widetilde{\mathcal{P}}}

\newcommand{\calE}{\mathcal{E}}
\newcommand{\calU}{\mathcal{U}}
\newcommand{\calV}{\mathcal{V}}
\newcommand{\calW}{\mathcal{W}}
\newcommand{\calX}{\mathcal{X}}
\newcommand{\calY}{\mathcal{Y}}
\newcommand{\calZ}{\mathcal{Z}}
\newcommand{\Triv}{\mathcal{T}}

\newcommand{\defeq}{\mathrel{\vcentcolon =}}

\newcommand{\comment}[1]{}

\newcommand{\tsup}[1]{\textsuperscript{#1}}

\newcommand{\realize}[1]{\lvert{#1}\rvert}
\newcommand{\VE}{VE}
\newcommand{\core}{\operatorname{core}}

\DeclareMathOperator{\F}{F}

\begin{document}

\title{The Brin--Thompson groups $sV$ are of type~$\F_\infty$}
\date{\today}
 % finiteness properties

\author[M.~G.~Fluch]{Martin G.~Fluch}
\address{Department of Mathematics, Bielefeld University,
  PO Box 100131, 33501 Bielefeld, Germany} 
\email{mfluch@math.uni-bielefeld.de}

\author[M.~Marschler]{Marco Marschler}
\address{Department of Mathematics, Bielefeld University,
  PO Box 100131, 33501 Bielefeld, Germany} 
\email{marco.marschler@math.uni-bielefeld.de}

\author[S.~Witzel]{Stefan Witzel}
\address{Mathematical Institute, University of M\"unster, 
  Einsteinstra\ss{}e 62, 48149 M\"unster, Germany}
\email{s.witzel@uni-muenster.de}

\author[M.~C.~B.~Zaremsky]{Matthew C.~B.~Zaremsky} 
\address{Department of Mathematical Science, Binghamton University,
  Binghamton, NY 13902, Unites States} 
\email{zaremsky@math.binghamton.edu}

\begin{abstract}
We prove that the Brin--Thompson groups~$sV$, also called higher
dimensional Thompson's groups, are of type~$\F_\infty$ for all
$s\in\N$.  This result was previously shown for~$s\le3$, by
considering the action of~$sV$ on a naturally associated space.  Our
key step is to retract this space to a subspace~$sX$ which is easier
to analyze.
\end{abstract}

\maketitle
\thispagestyle{empty}

% --------------------------------------------------------------------

\noindent Recall that a group is of \emph{type~$\F_\infty$} if it
admits a classifying space with finitely many cells in each dimension.
Well-known examples of groups of type~$\F_\infty$ include Thompson's
groups $F$, $T$\serialcomma and $V$.  Some generalizations of $V$ were
introduced by Brin~\cite{brin04,brin05} and shown to be simple.  We
denote these groups~$sV$, for $s\in\N$, with~$1V=V$.  These groups are
usually termed higher-dimensional Thompson's groups or Brin--Thompson
groups.  All of the groups $sV$ are known to be finitely presented
\cite{hennig12}, and Kochloukova, Martínez-Pérez\serialcomma and
Nucinkis \cite{kochloukova10} showed that~$2V$ and~$3V$ are of
type~$\F_\infty$.  We prove that this result extends to all
dimensions.

\begin{maintheorem}
\label{thrm:maintheorem}
The Brin--Thompson group $sV$ is of type~$\F_\infty$ for all $s$.
\end{maintheorem}

Fix some $s$.  There is a natural poset $\Poset_1$ associated to $sV$.
The realization $\realize{\Poset_1}$ of this poset is contractible and
the action of $sV$ is proper but not cocompact.  To prove the Main
Theorem it suffices to produce a cocompact filtration of
$\realize{\Poset_1}$ whose connectivity tends to infinity.  The tool
to study relative connectivity is discrete Morse theory.  This was
carried out for~$s=2,3$ in~\cite{kochloukova10}.  However, for
larger~$s$ this space quickly becomes cumbersome.

We therefore consider a subspace $sX$ of $\realize{\Poset_1}$ which we
call the \emph{Stein space} for~$sV$.  As before, the Stein space is
contractible and the action is not cocompact.  The advantage of the
Stein space is that the Morse theory becomes much easier to handle.

The paper is organized as follows.  In Section~\ref{sec:brin_thompson}
we recall the definition of $sV$.  The Stein space $sX$ is defined in
Section~\ref{sec:def_stein_space} and some basic properties are
verified.  In Section~\ref{sec:desc_link_conn} we analyze the
connectivity of the subspaces in the filtration and deduce the Main
Theorem.

\subsection*{Acknowledgments}

We would like to thank Kai-Uwe Bux for suggesting this project and for
many helpful discussions.  We also gratefully acknowledge support
through the SFB~701 in Bielefeld (first, second\serialcomma and fourth
author) and the SFB~878 in M\"unster (third author).  The research for
this article was carried out during the 2012 PCMI Summer Session in
Geometric Group Theory and we thank the organizers and the PCMI for
this opportunity.  Finally, we thank Matt Brin for his helpful
suggestions and remarks.

% --------------------------------------------------------------------

\section{The Brin--Thompson Groups}
\label{sec:brin_thompson}

The elements of the Brin--Thompson group $sV$ can be described as
dyadic self-maps of~$s$-dimensional cubes.  We will first give a brief
intuition for these maps, and then delve into some formalism.

To get an intuition for the elements of $sV$ for arbitrary~$s$, recall
first that elements of Thompson's group $V=1V$ can be thought of as
left-continuous, piecewise linear maps from the unit interval~$[0,1]$
to itself, where the slope of any linear piece is a positive dyadic
rational.  An equivalent description of such an element is obtained as
follows: first divide the unit interval representing the domain into
two halves and iterate this procedure by further subdividing some of
the resulting pieces.  Then similarly cut up the unit interval
representing the codomain into the same number of pieces as the
domain, and finally identify the pieces of the domain and codomain via
a permutation.  Note that the intervals identified in the last step
will usually have different lengths.  For more details
see~\cite{cannon96}.

To describe elements of~$sV$, we no longer think of the unit interval
but the unit~$s$-cube~$[0,1]^s$.  The unit $s$-cube can be halved by
dyadic hyperplanes in~$s$ different directions, as can any iterated
piece obtained this way.  As with $V$, an element of~$sV$ can be
described as a sequence of halvings of the domain and codomain and an
identification of the resulting pieces by a permutation.  Again the
identification will affinely deform the individual pieces.
Alternatively we can describe an element by a dyadic map from the
$s$-cube to itself.  A sequence of halvings of the~$s$-cube will be
modeled by ``dyadic coverings''.  To get an intuition, the reader
might want to look at Figure~\ref{fig:dyadic-map-2} (the map $f_1$
represents an element of~$2V$).  It may also be helpful to read
Section~1 of~\cite{burillo10}, which additionally details the
\emph{paired trees} model for elements of~$sV$.

\subsection{Dyadic maps and the group $sV$}

We now describe more formally the notions needed to define the group
$sV$, and also a certain poset $\Poset_1$, which will then be used to
define the space $sX$ for our main argument.

A real number is called \emph{dyadic} if it is of the form $k/2^\ell$
for some $k\in \Z$ and~$\ell \in \N_0$.  We denote by $\NDInterval$
the subspace of $[0,1]$ of non-dyadic numbers.  By a \emph{dyadic
interval} we mean a set of the form
$\bigl[\frac{k}{2^\ell},\frac{k+1}{2^\ell}\bigr] \cap \NDInterval$
with $k,\ell\in \N_0$, and the \emph{length} of the dyadic interval is
defined to be $1/2^{\ell}$.  A bijection $A \to B$ between dyadic
intervals is called a \emph{simple dyadic map} if it is affine of
positive slope.  Note that this slope will necessarily be a power of
two.

In general we consider the unit $s$-cube $I^s$ (or rather, the set of
non-dyadic points in the unit $s$-cube), which is the $s$-fold product
of~$I$.  A \emph{brick} is a subset $C$ of $I^s$ that is a product of
$s$ dyadic intervals, called the \emph{edges} of $C$, and the
\emph{volume} of $C$ is the product of the lengths of its edges.  Note
that the volume of a brick is always a power of two.  A \emph{dyadic
covering} is a finite set of bricks that disjointly
cover~$\NDInterval^s$ (note that by our definition the set $I$ does not
contain any dyadic numbers).

For a natural number $m$, denote by $I^s(m)$ the disjoint union
\begin{equation*}
I^s(m) \defeq B_1 \sqcup \cdots \sqcup B_m
\end{equation*}
where each $B_i$ is a copy of $I^s$.  Note that $I^s$ is the same
as~$I^s(1)$.  We call $B_i$ the \emph{$i$\tsup{th} block} of $I^s(m)$.
A covering $\calU$ of $I^s(m)$ is called \emph{dyadic} if it is a
disjoint union~$\calU = \calU_1\sqcup \cdots \sqcup \calU_m$
where~$\calU_i$ is a dyadic covering of the block~$B_i$.  We denote
by~$\Triv_m \defeq \{B_1, \ldots, B_m\}$ the \emph{trivial} dyadic
covering of $I^s(m)$, in which the bricks are just the blocks
themselves.

\begin{observation}
\label{obs:covering_lattice}
The set of dyadic coverings of $\NDInterval^s(m)$ is a lattice with
respect to the refinement relation.
\end{observation}

\begin{proof}
Existence of joins (that is, coarsest common refinements) as well as
existence of a unique minimum (namely, $\Triv_m$) are clear.  The
statement now follows from standard order theory.
\end{proof}

Let $\calU$ and $\calV$ be dyadic coverings of $I^s(m)$ and $I^s(n)$,
respectively, and let $f\: I^s(m) \to I^s(n)$ be a map.  We say that
the pair of dyadic coverings $(\calU, \calV)$ is \emph{compatible}
with $f$ if for every $C\in\calU$, $f|_C$ is a product of simple
dyadic maps and~$f(C) \in \calV$.  Less formally, this means that
every brick in the domain maps in an affine way to a brick in the
codomain.  If such a pair of dyadic coverings exists, then we say
that~$f$ is a \emph{dyadic map}.  It is easy to see that composition
of two dyadic maps is again a dyadic map, that every dyadic map is
invertible\serialcomma and that the inverse of a dyadic map is dyadic.

There is a combinatorial description of dyadic maps.  If $f \: I^s(m)
\to I^s(n)$ is a dyadic map and $(\calU_1, \calU_2)$ is a compatible
covering, then $f$ induces a bijection of dyadic coverings $\calU_1\to
\calU_2$.  Conversely, every bijection of dyadic coverings $\calU_1\to
\calU_2$ induces a dyadic map $I^s(m) \to I^s(n)$.

Note that two bijections $\calU_1 \to \calV_1$ and $\calU_2 \to
\calV_2$ induce the same map $I^s(m) \to I^s(n)$ if and only if there
are common refinements $\calU$ and $\calV$ such that the induced
bijections $\calU \to \calV$ coincide.

\begin{definition}
\label{def:sV}
The Brin--Thompson group $sV$ is the group of all dyadic self maps
of~$I^s$ with the multiplication given by composition, $gh \defeq
g\circ h$.
\end{definition}

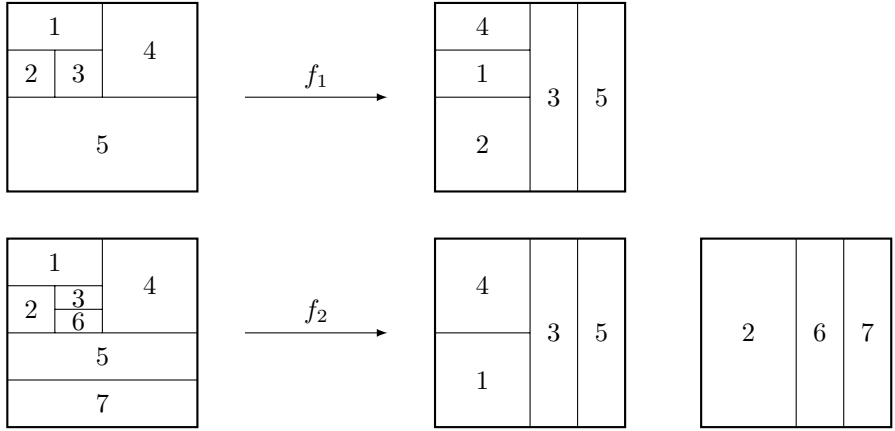
\begin{figure}
\centering
\begin{tikzpicture}[scale=1.25]
  \draw[thick] (0,0) rectangle (2,2);
  \draw
  (0,1) -- (2,1) 
  (1,1) -- (1,2)
  (0,1.5) -- (1,1.5) 
  (0.5,1) -- (0.5,1.5)
  
  (0.5,1.75)  node {1} 
  (0.25,1.25) node {2}
  (0.75,1.25) node {3}
  (1.5,1.5)   node {4}
  (1,0.5)     node {5};
  
  \draw[-latex] (2.5,1) -- (4,1) node[pos=0.5,above] {$f_1$};
  
  \draw[xshift=-.5cm, thick] (5,0) rectangle (7,2);
  \draw[xshift=-.5cm]
  (6,0) -- (6,2)
  (6.5,0) -- (6.5,2)
  (5,1) -- (6,1)
  (5,1.5) -- (6,1.5)
  
  (5.5,1.25) node {1}
  (5.5,0.5)  node {2}
  (5.5,1.75) node {4}
  (6.25,1)   node {3}
  (6.75,1)   node {5};
  
  \begin{scope}[yshift=-2.5cm]
  \draw[thick] (0,0) rectangle (2,2);
  \draw
  (0,1) -- (2,1) 
  (1,1) -- (1,2)
  (0,1.5) -- (1,1.5) 
  (0.5,1.25) -- (1,1.25)
  (0.5,1) -- (0.5,1.5)
  (0,0.5) -- (2,0.5);
  
  \draw
  (0.5,1.75)   node {1} 
  (0.25,1.25)  node {2}
  (0.75,1.37)  node {3}
  (0.75,1.125) node {6}
  (1.5,1.5)    node {4}
  (1,0.25)     node {7}
  (1,0.75)     node {5};
  
  \draw[-latex] (2.5,1) -- (4,1)node[pos=0.5,above] {$f_2$};
  
  \draw[xshift=-.5cm, thick] (5,0) rectangle (7,2);
  \draw[xshift=-.5cm]
    (6,0) -- (6,2)
    (6.5,0) -- (6.5,2)
    (5,1) -- (6,1)
  
   (5.5,1.5) node {4}
   (5.5,0.5) node {1}
   (6.25,1)  node {3}
   (6.75,1)  node {5};
  
  \draw[xshift=-.7cm, thick] (8,0) rectangle (10,2);
  \draw[xshift=-.7cm]
  (9,0) -- (9,2)
  (9.5,0) -- (9.5,2)
 
  (8.5,1)  node {2}
  (9.25,1) node {6}
  (9.75,1) node {7};
  \end{scope}
\end{tikzpicture}

\caption{An example of dyadic map $f_1\: I^2(1) \to I^2(1)$ and a
dyadic map $f_2\: I^2(1) \to I^2(2)$.  Note that $f_2$ is obtained
from~$f_1$ by splitting along a horizontal line.  The map $f_2$ is
equivalent in $\Poset_1$ to the one where the blocks on the right are
interchanged.}
\label{fig:dyadic-map-2}
\end{figure}

\subsection{The poset $\Poset_1$}

In order to define the poset $\Poset_1$ on which $sV$ acts we need
some more notation.

Denote by~$\prePoset_{m,n}$ the set of all dyadic maps $f\: I^s(m)\to
I^s(n)$, so for example~$\prePoset_{1,1}=sV$.  Let $\prePoset$ be the
union of the~$\prePoset_{m,n}$ where $m$ and $n$ range over all
positive integers.  Also denote by~$\prePoset_m$ the subset
of~$\prePoset$ where the domain of the maps consists of $m$ blocks.

There is a natural action of~$sV$ on $\prePoset_1$ given by
precomposition: $f^g \defeq f\circ g$ for~$g \in sV$ and~$f \in
\prePoset_1$.  For each positive $n$ there is also an action of the
symmetric group $S_n$ on~$\prePoset_{m,n}$ by permuting the blocks of
the codomain.  We denote the quotient~$\prePoset_{m,n} / S_n$
by~$\Poset_{m,n}$.  In other words, an element of $\Poset_{m,n}$ is
obtained from~$\prePoset_{m,n}$ by forgetting the order of the blocks
in the codomain.  We set
\begin{equation*}
\Poset \defeq \bigcup_{n,m\ge 1} \Poset_{m,n} \qquad \text{and}
\qquad \Poset_1 \defeq \bigcup_{n\ge 1} \Poset_{1,n} \text{\,.}
\end{equation*}

Note that $\prePoset_{1,n}$ is an $sV$-invariant subset of
$\prePoset_1$, and the action of $sV$ on $\prePoset_{1,n}$ commutes
with the action of the symmetric group~$S_n$, so we get an action
of~$sV$ on~$\Poset_{1,n}$ for every $n$.  In particular the
$sV$-action on~$\prePoset_1$ induces an action of~$sV$ on~$\Poset_1$.

\begin{definition}
\label{def:t}
The function $t\: \Poset \to \N$ assigns to each $x\in \Poset$ the
number of blocks in the codomain of~$x$, that is, if $x\in
\Poset_{m,n}$ for some $m$, then $t(x)=n$.
\end{definition}

Next we define a poset structure on $\Poset$ using the notion of
``splitting''.  A dyadic map $z\: I^s(m) \to I^s(n)$ is called a
\emph{splitting (along $\calU$)} if~$z$ is compatible with a pair of
dyadic coverings of the form $(\calU, \Triv_n)$.  The splitting $z$ is
called \emph{non-trivial} if~$n>m$.  Colloquially then, as the name
implies, a non-trivial splitting is given by splitting up some cubes
(and then not sticking any resulting cubes together).  The inverse of
a splitting (along $\calU$) is called a \emph{merging} (along
$\calU$).

We define an order~$\le$ on $\prePoset$ by saying that~$x< y$ if
there exists a non-trivial splitting~$z$ such that $y=z\circ x$, that
is, if~$y$ is obtained from $x$ by non-trivial splitting.  We also
denote the induced order on $\Poset$ by $\le$.  In particular,
$\Poset_1$ is ordered by $\le$.  See Figure~\ref{fig:dyadic-map-2} for
an example of dyadic maps and splitting.

The poset $\Poset_1$ is filtered by the $t$-sublevel sets
\begin{equation*}
\Poset_1^{\le n} = \bigcup_{1 \le k \le n} \Poset_{1,k} \text{\,.}
\end{equation*}

We make the following easy observations.

\begin{observation}
\label{obs:big_space_contractible}
The poset $\prePoset_1$ is directed (that is, any two elements have a
common upper bound).  Therefore, $\realize{\prePoset_1}$ and
$\realize{\Poset_1}$ are contractible.
\end{observation}

\begin{observation}
\label{obs:stabilizers}
The action of $sV$ on $\prePoset_1$ is free.  Thus, for any vertex $x$
in $\realize{\Poset_1}$, the stabilizer $\Stab_{sV}(x)$ is finite.
Hence all cell stabilizers are finite and of type~$\F_\infty$.
\end{observation}

\begin{observation}
\label{obs:cocompactness}
The action of $sV$ on $\Poset_{1}^{= 1}$ is transitive, and for each
$n\ge 1$ the sublevel set $\realize{\Poset_{1}^{\le n}}$ is locally
finite.  Hence $\realize{\Poset_{1}^{\le n}}$ is finite modulo $sV$.
\end{observation}

These observations suggest that the filtration
$(\realize{\Poset_1^{\le n}})_n$ of $\realize{\Poset_1}$ could be used
to show that~$sV$ is of type~$\F_\infty$, using Brown's Criterion.

\begin{brownscriterion}
\cite[Corollary~3.3]{brown87}
\label{brownscriterion}
Let $G$ be a group and $X$ a contractible $G$-CW-complex such that the
stabilizer of every cell is of type~$\F_\infty$.  Let $\{X_{j}\}_{j\ge
1}$ be a filtration of $X$ such that each $X_{j}$ is finite
$\textnormal{mod}~G$.  Suppose that the connectivity of the pair
$(X_{j+1},X_j)$ tends to~$\infty$ as~$j$ tends to~$\infty$.  Then $G$
is of type~$\F_\infty$.
\end{brownscriterion}

It would suffice now to show that the connectivity of the pair
$(\realize{\Poset_1^{\le n+1}},\realize{\Poset_1^{\le n}})$ tends
to~$\infty$ as~$n$ tends to~$\infty$.  This was proved by Kochloukova,
Martínez-Pérez\serialcomma and Nucinkis~\cite{kochloukova10} for the
cases $s=2,3$.  However, it becomes increasingly difficult to verify
for higher~$s$.  The main difference of our approach here is that we
consider a certain subcomplex~$sX$ of $\realize{\Poset_1}$.  Analyzing
the relative connectivity in~$sX$ turns out to be substantially easier
than in $\realize{\Poset_1}$.

% --------------------------------------------------------------------

\section{The Stein space for $sV$}
\label{sec:def_stein_space}

The idea of passing to what we are calling a ``Stein space'' was first
introduced by Stein~\cite{stein92}, and in particular was used to
obtain a new proof that~$F$ is of type~$\F_\infty$.  This construction
generalizes nicely to deal with some complicated versions of
Thompson's groups.  For example Stein spaces were used in~\cite{bux12}
to prove that braided Thompson's groups are of type~$\F_\infty$.  The
key idea is that the splitting establishing a relation~$x\le y$ can be
obtained from ``elementary splittings'' that give rise to elementary
relations $x\preceq x_1\preceq\ldots\preceq x_r\preceq y$, and these
small steps are much easier to understand locally.  Heuristically, an
elementary splitting amounts to halving an~$s$-cube at most once in
any given direction.  We now describe more rigorously the construction
of the Stein space.

\begin{definition}
\label{def:elem_split}
Call a brick $C$ \emph{elementary} if every edge of $C$ has length at
least~$1/2$.  Call an elementary brick \emph{very elementary} if it
has volume at least~$1/2$.  A~dyadic covering $\calU$ is called
\emph{(very) elementary} if every brick in $\calU$ has this property.
Likewise, a splitting or merging along $\calU$ is \emph{(very)
elementary} if~$\calU$ is.
\end{definition}

For $x,y\in \Poset$, if $y$ can be obtained from $x$ by an elementary
splitting, write $x\preceq y$; if moreover $x\neq y$ then we write
$x\prec y$.  If $y$ is obtained from $x$ by a very elementary
splitting, write $x\sqsubseteq y$; if moreover $x\neq y$, then we
write $x\sqsubset y$.  Note that the relations $\preceq$ and
$\sqsubseteq$ are not transitive.  In particular, the length of a
chain of very elementary splittings is bounded by the number of
blocks.  However, if $x_1\le x_2\le x_3$ and $x_1\preceq x_3$ then
$x_1\preceq x_2$ and $x_2\preceq x_3$, and analogously
for~$\sqsubseteq$.  It is clear that the action of $sV$ respects the
relations $\le$, $\preceq$ and~$\sqsubseteq$.

Clearly $I^s(m)$ has a unique maximal elementary covering $\calE$ by
$m \cdot 2^s$ bricks all of which have volume $2^{-s}$.  A covering is
elementary if and only if $\calE$ is a refinement of it.

The closed interval $[x,y]$ in $\Poset_1$ is defined to be $[x,y]
\defeq \{w\in \Poset_1\mid x\le w\le y\}$; the open and half-open
intervals are defined analogously.  Call an interval~$[x,y]$
in~$\realize{\Poset_1}$ \emph{elementary} if $x\preceq y$, and
\emph{very elementary} if $x\sqsubseteq y$.  A simplex of
$\realize{\Poset_1}$ is \emph{(very) elementary} if there is a (very)
elementary interval that contains all of its vertices.

\begin{definition}
\label{def:stein-space}
The \emph{Stein space for $sV$}, denoted $sX$, is the subcomplex
of~$\realize{\Poset_1}$ consisting of elementary simplices.
\end{definition}

The following statement is the key to showing the contractibility of
the Stein space.

\begin{lemma}\label{lem:elem_core}
Let $x,y\in \Poset_1$ with $x\le y$.  There exists a unique $y_0\in
[x,y]$ such that~$x\preceq y_0$ and for any $x\preceq w\le y$, we
have $w\le y_0$.  If $x < y$, then $x < y_0$.
\end{lemma}

\begin{proof}
Set $m \defeq t(x)$ and $n\defeq t(y)$.  Let $\tilde{x}$ be a
representative in $\prePoset_1$ for $x$.  Let~$\calU$ be the dyadic
covering of $I^s(m)$ such that $y$ is obtained from $\tilde{x}$ by
splitting along $\calU$.  Let~$\calE$ be the maximal elementary
covering of $I^s(m)$.  The element $y_0$ is obtained from $\tilde{x}$
by splitting along the finest common coarsening $\calE \wedge \calU$.
The desired properties follow from
Observation~\ref{obs:covering_lattice}.
\end{proof}

For $x\le y$, call the $y_0$ from the lemma the \emph{elementary core
of $y$ with respect to~$x$}, and denote it $\core_x(y) \defeq y_0$.
When $x$ is understood we omit the subscript.  Observe that if
$y_1\le y_2$ then $\core(y_1)\le \core(y_2)$, that is, taking
elementary cores respects the poset relation.
Figure~\ref{fig:elem_core} gives an example of an elementary core.

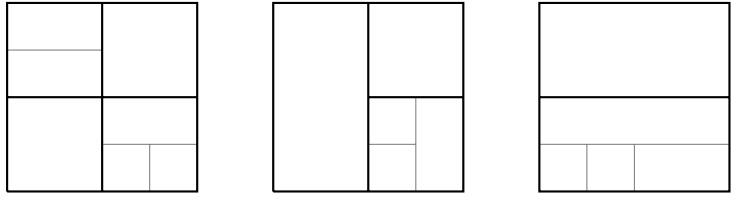
\begin{figure}
 \centering
 \begin{tikzpicture}[scale=1.25]
  \draw[gray]
    (0,0) -- (2,0) -- (2,2) -- (0,2) -- (0,0)   
    (1,0) -- (1,2)   (0,1) -- (2,1) 
    (0,1.5) -- (1,1.5)   (1,0.5) -- (2,0.5)
    (1.5,0.5) -- (1.5,0);
   \draw[thick]
    (0,0) -- (2,0) -- (2,2) -- (0,2) -- (0,0)
    (1,0) -- (1,2)
    (0,1) -- (2,1);
   \begin{scope}[xshift=2.8cm]
   \draw[gray]
    (0,0) -- (2,0) -- (2,2) -- (0,2) -- (0,0)
    (1,0) -- (1,2)
    (1,1) -- (2,1)
    (1.5,0) -- (1.5,1)   
    (1,0.5) -- (1.5,0.5);
   \draw[thick]
    (0,0) -- (2,0) -- (2,2) -- (0,2) -- (0,0)
    (1,0) -- (1,2) 
    (1,1) -- (2,1);
   \end{scope}
   \begin{scope}[xshift=5.6cm]
   \draw[gray]
    (0,0.5) -- (2,0.5)  (1,0) -- (1,0.5)
    (0.5,0) -- (0.5,0.5);
   \draw[thick] 
    (0,0) rectangle (2,2)  (0,1) -- (2,1);
   \end{scope}
 \end{tikzpicture}
\caption{A non-elementary dyadic covering, for $s=2$.  The thick lines
indicate the elementary core.}
\label{fig:elem_core}
\end{figure}

\begin{lemma}
\label{lem:cube_lemma}
For $x<y$ with $x\not\prec y$, $|(x,y)|$ is contractible.
\end{lemma}

The proof is essentially the same as the proof of the lemma in
Section~4 of~\cite{brown92}.

\begin{proof}
If $w\in (x,y]$, then $\core(w) \in [x,y)$ because $x \not\prec y$, and
$\core(w) \in (x,y]$ because $x < w$. So in fact $\core(w)\in(x,y)$. Also,
$\core(w)\le \core(y)$ by the previous discussion. The inequalities $w\ge
\core(w)\le \core(y)$ provide a contraction of~$|(x,y)|$, by Section~1.5
of~\cite{quillen78}.
\end{proof}

As was done in~\cite{brown92} for the Stein space of~$V$, we can build
up from $sX$ to~$\realize{\Poset_1}$ to show that $sX$ is
contractible.

\begin{corollary}
\label{cor:stein_space_cible}
The Stein space $sX$ is contractible for all $s$.
\end{corollary}

\begin{proof}
By Observation~\ref{obs:big_space_contractible}, $\realize{\Poset_1}$
is contractible.  We build up from $sX$ to $\realize{\Poset_1}$
without changing the homotopy type.

Given a closed interval $[x,y]$, define $r([x,y]) \defeq t(y)-t(x)$.
We attach the contractible subcomplexes $|[x,y]|$ for $x\not\prec y$
to $sX$ in increasing order of $r$-value.  When we attach $|[x,y]|$
then, we attach it along $|[x,y)|\cup|(x,y]|$.  But this is the
suspension of $|(x,y)|$, and so is contractible by the previous lemma.
We conclude that attaching $|[x,y]|$ does not change the homotopy
type, and since $\realize{\Poset_1}$ is contractible, so is~$sX$.
\end{proof}

\medskip

For each $n\ge1$ let $sX^{\le n}$ be the full subcomplex of $sX$
spanned by vertices~$x$ with~$t(x)\le n$.  Similarly define
$sX^{<n}$, and let $sX^{=n}$ be the set of vertices $x$ with~$t(x)=n$.
Note that all of these sets are invariant under the action of $sV$.
We will show that the filtration $(sX^{\le n})_n$ of $sX$ satisfies
the assumptions of Brown's Criterion.

Thanks to Observations~\ref{obs:stabilizers}
and~\ref{obs:cocompactness}, and
Corollary~\ref{cor:stein_space_cible}, the only remaining feature of
the filtration $(sX^{\le n})_{n}$ of $sX$ that we need to verify is
that the connectivity of the pair $(sX^{\le n+1},sX^{\le n})$ tends
to~$\infty$ as~$n$ tends to~$\infty$.  This is exactly the condition
that proved difficult to verify for the filtration of
$\realize{\Poset_1}$ in~\cite{kochloukova10}.

We will verify the relative connectivity in the next section using
discrete Morse theory.  The idea is to treat $t$ as a height function
on $sX$ and inspect descending links.

% --------------------------------------------------------------------

\section{Connectivity of the descending links and proof of the Main
Theorem}
\label{sec:desc_link_conn}

We will use the following Morse-theoretic tools.  Fix a vertex~$x$
in~$sX$, say with~$t(x)=n$, and call~$n$ the \emph{height} of~$x$.
The \emph{descending link}~$\dlk(x)$ of~$x$ is defined to be the
intersection of~$\lk(x)$ with~$sX^{< n}$.  The fact that vertices with
equal heights cannot share an edge means that we can obtain~$sX^{\le
n}$ from $sX^{<n}$ by ``gluing in'' each vertex at height~$n$ along
its descending link.  This is made rigorous by the Morse lemma
(cf.~Corollary~2.6 of~\cite{bestvina97}):

\begin{lemma}
\label{lem:morse}
Let~$X$ be a simplicial complex and let~$f \colon X^{(0)} \to \Z$ be
such that $f(x) \ne f(y)$ for adjacent vertices~$x$ and~$y$ of~$X$.
If~$\dlk(x)$ is~$(k-1)$-connected for every vertex~$x \in X^{=n}$, then
the pair $(X^{\le n},X^{<n})$ is $k$-connected, that is, the
inclusion~$X^{<n} \hookrightarrow X^{\le n}$ induces an isomorphism in
$\pi_j, j < k$ and an epimorphism in $\pi_k$.
\end{lemma}

Fix a vertex~$x$ in~$sX$ and consider~$L(x)\defeq\dlk(x)$.  As a
subcomplex of~$\realize{\Poset_1}$,~$L(x)$ is the collection of
simplices given by chains~$y_k<\ldots< y_0< x$ with~$y_k\prec x$.  We
first consider the subcomplex~$L_0(x)$ of~$L(x)$ consisting of such
chains with $y_k\sqsubset x$.

The complex~$L_0(x)$ naturally projects onto a matching complex.

\begin{definition}
Let~$\Gamma$ be a graph.  The \emph{matching complex}~$\M(\Gamma)$
of~$\Gamma$ is the simplicial complex with a~$k$-simplex for every
collection~$\{e_0, \dots, e_k\}$ of~$k+1$ pairwise disjoint edges,
with the face relation given by inclusion.  If we consider oriented
edges, we get the \emph{oriented matching complex}~$\M^o(\Gamma)$.
\end{definition}

The specific graphs that we will need are generalizations of complete
graphs.  For~$s\in\N$, let~$sK_n$ be the graph with~$n$ nodes and~$s$
edges between any two distinct nodes.  In particular~$1K_n$ is
just~$K_n$, the complete graph on~$n$ nodes.  Color the edges from~$1$
to~$s$ so that any two distinct nodes have precisely one edge of each
color between them.  For a fixed labeling~$1$ through~$n$ of the nodes
of each~$sK_n$, we have a projection $s\pi\: sK_n \rightarrow K_n$ for
each~$s$, given by sending an edge with endpoints~$i$ and~$j$ to the
unique edge of~$K_n$ with endpoints~$i$ and~$j$.  Since disjoint edges
map to disjoint edges, this induces a map $\M(s\pi)\: \M(sK_n)
\rightarrow \M(K_n)$.

For any~$\ell\in\Z$, define $\nu(\ell) \defeq \bigl\lfloor
\frac{\ell-2}{3} \bigr\rfloor$.

\begin{lemma}\label{lem:match_cpx_conn}
$\M(sK_n)$ is~$(\nu(n)-1)$-connected, as is~$\M^o(sK_n)$.
\end{lemma}

\begin{proof}
It is well known that~$\M(K_n)$ is~$(\nu(n)-1)$-connected, see for
example~\cite{athanasiadis04, bux12, bjorner94}.  For
any~$k$-simplex~$\sigma$ in~$\M(K_n)$, the fiber
$\M(s\pi)^{-1}(\sigma)$ is the join of the fibers of the vertices
of~$\sigma$, so in particular is homotopy equivalent to a wedge of
spheres of dimension~$k$.  Moreover, it is clear that links in $\M(K_n)$ are
themselves matching complexes of complete graphs. Therefore the hypotheses of
Theorem~9.1 in~\cite{quillen78} are satisfied, and we conclude that~$\M(sK_n)$
is~$(\nu(n)-1)$-connected. We also have an obvious map~$\M^o(sK_n)
\twoheadrightarrow \M(sK_n)$ obtained by forgetting the orientation on the
edges. The fibers of this map are similarly spherical of the right dimension, so
again using Theorem~9.1 of~\cite{quillen78} we conclude that~$\M^o(sK_n)$
is~$(\nu(n)-1)$-connected.
\end{proof}

Every vertex~$y \in L_0(x)$, say with~$t(y) = m$, is obtained from~$x$
by applying a non-trivial very elementary merging.  The merging is
given by a very elementary covering~$\calU$ of~$m$ blocks whose~$n$
bricks are indexed by the blocks of~$x$.  Two such mergings define the
same element~$y$ if and only if they differ by a permutation of the
blocks.  Consequently, denoting by~$\VE_n$ the set of very elementary
coverings by~$n$ labeled bricks up to permutation of the blocks, we
get a one-to-one correspondence between~$L_0(x)$ and~$\VE_n$.  We
obtain a partial order~$\VE_n$ from the partial order on~$\Poset_1$
via this identification.

\begin{corollary}\label{cor:very_desc_lk_conn}
$\VE_n$, and therefore~$L_0(x)$, is isomorphic to~$\M^o(sK_n)$.
Hence, both are $(\nu(n)-1)$-connected.
\end{corollary}

\begin{proof}
Consider a non-trivial very elementary dyadic covering~$\calU$
of~$I^s(m)$ with~$n$ bricks labeled~$1$ to~$n$.  Since~$\calU$ is very
elementary, each block consists of at most two bricks.  If it does
consist of two bricks, then it defines an oriented edge in the
graph~$sK_n$ as follows.  The two bricks are
\begin{equation*}
\NDInterval^{k-1} \times \biggl(\NDInterval \cap
\Bigl[0,\frac{1}{2}\Bigr]\biggr) \times \NDInterval^{s-k}
\quad \text{and} \quad
\NDInterval^{k-1} \times \biggl(\NDInterval \cap
\Bigl[\frac{1}{2},1\Bigr]\biggr) \times \NDInterval^{s-k}
\end{equation*}
for some~$1 \le k \le s$.  Say the first brick is labeled~$i$ and the
second brick is labeled~$j$.  Then the edge in~$sK_n$ defined by this
block points from~$i$ to~$j$ and has color~$k$.

See Figure~\ref{fig:map-VE-to-Mo} for an example.
\begin{figure}
\centering

\begin{tikzpicture}
\begin{scope}[xshift=-2cm,scale=1.6]
\draw[thick] (1,-.35) rectangle +(0.7,0.7)
      (2,-.35) rectangle +(0.7,0.7)
      (3,-.35) rectangle +(0.7,0.7);
\draw (1,-.35) rectangle +(0.7,0.35) node[pos=0.5] {$5$}
      (1,0) rectangle +(0.7,0.35) node[pos=0.5] {$2$}
      (2,-.35) rectangle +(0.7,0.7) node[pos=0.5] {$4$}
      (3,-.35) rectangle +(0.35,0.7) node[pos=0.5] {$1$}
      (3.35,-.35) rectangle +(.35,0.7) node[pos=0.5] {$3$};
\end{scope}

\draw[|->] (4.6,0) -- (6.5,0) node[pos=0.5, above] {$\pi$};

\begin{scope}[xshift=8.9cm]
\coordinate (A) at (90:1.5);
\coordinate (B) at (162:1.5);
\coordinate (C) at (234:1.5);
\coordinate (D) at (306:1.5);
\coordinate (E) at (18:1.5);

\draw[gray!30,thick] (D) -- (A) -- (B) -- (C) -- (D) 
  -- (E) -- (A) -- (C) -- (E) -- (B) -- cycle;

\filldraw (A) circle (1pt) node[anchor=south] {$1$}
          (B) circle (1pt) node[anchor=south east] {$2$}
          (C) circle (1pt) node[anchor=north east] {$3$}
          (D) circle (1pt) node[anchor=north west] {$4$}
          (E) circle (1pt) node[anchor=south west] {$5$};
          
\draw[thick,shorten <=4pt,shorten >=4pt,-latex] (A) -- (C);          
\draw[thick,shorten <=4pt,shorten >=4pt,-latex,dashed] (E) -- (B);
\end{scope}
\end{tikzpicture}
\caption{An example of~$\pi\: \VE_n \to \M^o(sK_n)$ in the case~$n=5$
and~$s=2$.  The solid arrow corresponds to a merge along a vertical
face, and the dashed arrow corresponds to a merge along a horizontal
face.}
\label{fig:map-VE-to-Mo}
\end{figure}
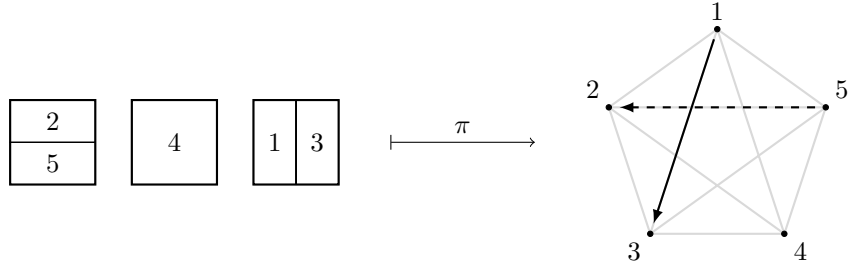

This procedure defines an isomorphism of ordered sets
$\VE_n\rightarrow\M^o(sK_n)$.
The connectivity statement now follows from
Lemma~\ref{lem:match_cpx_conn}.
\end{proof}

The next step is to show that~$L(x)$ is highly connected by building
up from~$L_0(x)$ to~$L(x)$ along highly connected links.  If~$s = 1$,
then $L_0(x) = L(x)$ so we may assume~$s > 1$ in what follows.

We start by giving a combinatorial description of~$L(x)$ similar to
the one given for~$L_0(x)$ before.  Every vertex in~$L(x)$ is obtained
from~$x$ via a non-trivial elementary merging.  We can therefore
replace ``very elementary'' by ``elementary'' in the discussion
of~$\VE_n$ above.  We get that the poset~$E_n$ of elementary mergings
of~$n$ labeled bricks up to permutation of blocks is isomorphic
to~$L(x)$.

We now describe the Morse function that determines in which order we
build up from~$L_0(x)$ to~$L(x)$.  For any~$\calU\in E_n$, the volume
of any brick in~$\calU$ is at least~$1/2^s$.  For each~$0\le i\le s$
define~$c_i$ to be the number of bricks in~$\calU$ with
volume~$1/2^i$.  Then define~$c$ to be the lexicographically ordered
function $c=(c_s,c_{s-1},\dots,c_3,c_2)$.  Note that we \emph{do~not}
include~$c_1$ or $c_0$ in this tuple; this will be crucial to our
arguments.  Denote by~$b$ the number of blocks of~$\calU$.  The
\emph{height}~$h$ of~$\calU$ is now defined to be~$h=(c,b)$, ordered
lexicographically.

\begin{observation}\label{obs:ht_fxn_rules}
Let~$\calX$ and~$\calY$ be vertices in~$E_n$ with~$\calX<\calY$.  Then
$c(\calX)\ge c(\calY)$ and~$b(\calX)<b(\calY)$, so in particular
$h(\calX)<h(\calY)$ if and only if~$c(\calX)=c(\calY)$,
and~$h(\calX)>h(\calY)$ if and only if~$c(\calX)>c(\calY)$.
\end{observation}

Fix a vertex~$\calU$ in~$E_n\setminus \VE_n$.  The descending link
of~$\calU$ with respect to~$h$ will be denoted~$\dlk_h(\calU)$.  There
are two types of vertices $\calV$ in~$\dlk_h(\calU)$.  First, we could
have~$\calU>\calV$ and $h(\calU)>h(\calV)$, which by the above
observation implies that~$c(\calU)=c(\calV)$.  The full subcomplex
of~$\dlk_h(\calU)$ spanned by such vertices will be called the
\emph{(descending) down-link}.  Second, we could have $\calU<\calV$
and~$h(\calU)>h(\calV)$, which implies that~$c(\calU)>c(\calV)$.  The
full subcomplex of~$\dlk_h(\calU)$ spanned by these vertices will be
called the \emph{(descending) up-link}.

\begin{observation}
Vertices~$\calV$ in the down-link and~$\calW$ in the up-link
automatically satisfy~$\calV<\calW$.  Therefore~$\dlk_h(\calU)$ is a
join of the down-link and the up-link.
\end{observation}

This allows us to study the up-link and the down-link separately.

\begin{lemma}
\label{lem:two_bricks}
If~$\calU$ has a block with precisely two bricks, then the up-link
of~$\calU$ is contractible, and hence so is~$\dlk_h(\calU)$.
\end{lemma}

\begin{proof}
Let~$B$ be a block in~$\calU$ with two bricks.  Note that splitting
only~$B$ does not yield a vertex with lower height.  For a
vertex~$\calV$ of the up-link we define a vertex~$\calV_0$ as follows
(see Figure~\ref{fig:building_up_the_link}).  Since~$\calV$ is in the
up-link, it is obtained from~$\calU$ by splitting.  Let~$\calV_0$ be
the covering obtained from~$\calU$ by doing all the same splittings as
for~$\calV$, except that~$B$ is not split (whether or not it was split
for~$\calV$).  Then~$\calV_0>\calU$, since~$\calV$ was obtained by
splitting more than just~$B$, as observed above.  It is also clear
that~$c(\calV_0)<c(\calU)$, and so~$\calV_0$ is again in the up-link
of~$\calU$.  Now let~$\calZ_B$ be the maximal elementary splitting
of~$\calU$ that does not split~$B$.  Then for all~$\calV$ in the
up-link, we have~$\calV_0\le\calZ_B$.  Hence we have the
inequalities~$\calV\ge\calV_0\le\calZ_B$, which provide a
contraction of the up-link of~$\calU$, by Section~1.5
of~\cite{quillen78}.
\end{proof}

\begin{figure}
\centering
\begin{tikzpicture}[scale=0.5]
\begin{scope} 
\filldraw[xshift=5cm,gray!10] (1.5,-1) rectangle (3.5,1);
\draw[thick,xshift=5cm] (-1,-1) rectangle (1,1)
      (1.5,-1) rectangle (3.5,1);
\draw[xshift=5cm] (-1,0) -- (1,0)
          (0,-1) -- (0,1)
          (2.5,-1) -- (2.5,1)
          
          (-.5,.5) node {1}
          (.5,.5) node {2}
          (-.5,-.5) node {3}
          (.5,-.5) node {4}
          (2,0) node {5}
          (3,0) node {6}

          (-2cm,0) node {$\calU$};
\end{scope}

\filldraw[xshift=-1cm,yshift=-3cm,gray!10]
      (4,-1) rectangle (6,1);
\draw[thick,xshift=-1cm,yshift=-3cm] (-1,-1) rectangle (1,1)
      (1.5,-1) rectangle (3.5,1)
      (4,-1) rectangle (6,1);

\draw[xshift=-1cm,yshift=-3cm]
      (2.5,-1) -- (2.5,1)
      (0,-1) -- (0,1)
      (5,-1) -- (5,1)

      (-.5,0) node {1}
      (.5,0) node {2}
      (2,0) node {3}
      (3,0) node {4}
      (4.5,0) node {5}
      (5.5,0) node {6}
      (-2cm,0) node {$\calV_0$};
        
\filldraw[xshift=7cm,yshift=-6cm,gray!10]
      (4,-1) rectangle (6,1)
      (6.5,-1) rectangle (8.5,1);
\draw[thick,xshift=7cm,yshift=-6cm] (-1,-1) rectangle (1,1)
      (1.5,-1) rectangle (3.5,1)
      (4,-1) rectangle (6,1)
      (6.5,-1) rectangle (8.5,1);
\draw[xshift=7cm,yshift=-6cm]
      (0,-1) -- (0,1)
      (2.5,-1) -- (2.5,1)

      (-.5,0) node {1}
      (.5,0) node {2}
      (2,0) node {3}
      (3,0) node {4}
      (5,0) node {5}
      (7.5,0) node {6}
      (-2cm,0) node {$\calV$};
       
\filldraw[xshift=-3.5cm,yshift=-9cm,gray!10]
      (9,-1) rectangle (11,1);
\draw[thick,xshift=-3.5cm,yshift=-9cm] (-1,-1) rectangle (1,1)
      (1.5,-1) rectangle (3.5,1)
      (4,-1) rectangle (6,1)
      (6.5,-1) rectangle (8.5,1)
      (9,-1) rectangle (11,1);
\draw[xshift=-3.5cm,yshift=-9cm]
       (10,-1) -- (10,1)

       (0,0) node {1}
       (2.5,0) node {2}
       (5,0) node {3}
       (7.5,0) node {4}
       (9.5,0) node {5}
       (10.5,0) node {6}
       (-2cm,0) node {$\calZ_B$};

\draw (6.25,-1.15) -- (1.5,-1.85)
      (6.25,-1.15) -- (10.75,-4.85)
      (1.5,-4.15) -- (1.5,-7.85)
      (1.5,-4.15) -- (10.75,-4.85);
\end{tikzpicture}
\caption{A step in building up from~$\VE_6$ to~$E_6$ as described in
the proof of Lemma~\ref{lem:two_bricks}.  The block~$B$ of the
covering $\calU$ and its images under the various splittings are
highlighted.}
\label{fig:building_up_the_link}
\end{figure}
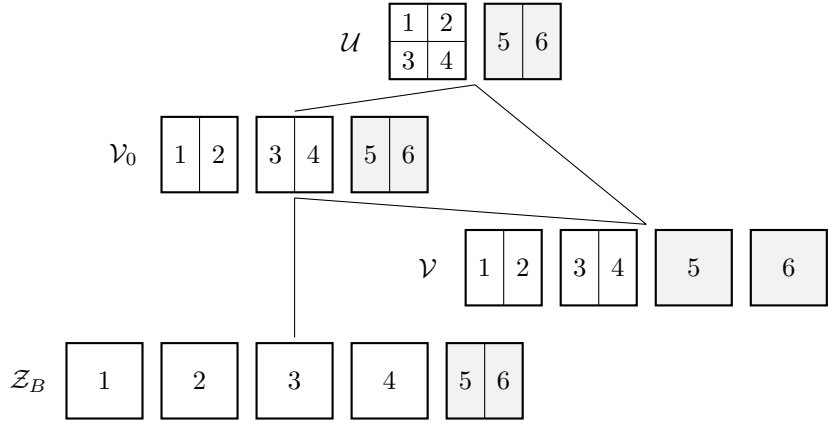

For~$\ell\in\Z$, define
$\eta(\ell)\defeq\bigl\lfloor\frac{\ell-2}{2^s}\bigr\rfloor$.  Note
that, for a fixed~$s$, as $n\rightarrow\infty$,~$\eta(n)$ increases
monotonically to~$\infty$.

\begin{lemma}\label{lem:no_two_bricks}
If~$\calU$ has no block with precisely two bricks,
then~$\dlk_h(\calU)$ is at least $(\eta(n)-2)$-connected.
\end{lemma}

\begin{proof}
Call a block in~$\calU$ with more than two bricks \emph{big}, and a
block with only one brick \emph{small}.  Let~$k_b$ be the number of
big blocks and~$k_s$ the number of small blocks.  By
assumption~$k_b+k_s$ is the number~$m$ of blocks in~$\calU$.

The up-link of~$\calU$ is clearly at least~$(k_b-2)$-connected, since
splitting a big block in any way produces a vertex with lower height,
and so each big block contributes a non-empty join factor to the
up-link.  The down-link of~$\calU$ is isomorphic to~$\VE_{k_s}$, and
therefore is~$(\nu(k_s)-1)$-connected by
Corollary~\ref{cor:very_desc_lk_conn}.  This implies
that~$\dlk_h(\calU)$ is~$(k_b+\nu(k_s)-1)$-connected.  Also,~$n$ is
the number of bricks in~$\calU$, so~$n\le 2^sk_b+k_s$.

Since~$s>1$,~$2^s>3$, so we have 
\begin{multline*}
k_b+\nu(k_s)-1
\ge 
k_b + \biggl\lfloor \frac{k_s-2}{2^s} \biggr\rfloor-1
\ge k_b +
\frac{k_s-2}{2^s}-2 \\
=  \frac{2^sk_b+k_s-2}{2^s} - 2 \ge
\frac{n-2}{2^s}-2\ge\eta(n)-2 \text{\,.}
\end{multline*}
We conclude that $\dlk_h(\calU)$ is at least~$(\eta(n)-2)$-connected.
\end{proof}

\begin{corollary}\label{cor:desc_lk_conn}
If~$s=1$ then~$E_n$, and hence~$L(x)$ is~$(\nu(n)-1)$-connected. If~$s>1$,
then~$E_n$, and hence~$L(x)$ is~$(\eta(n)-1)$-connected.
\end{corollary}

\begin{proof}
The~$s=1$ case is already done, since then~$E_n=\VE_n$.  Now
suppose~$s>1$.  Then $\eta\le\nu$, so~$\VE_n$ is at
least~$(\eta(n)-1)$-connected.  Also, for~$\calU\in E_n\setminus
\VE_n$,~$\dlk_h(\calU)$ is $(\eta(n)-2)$-connected by
Lemmas~\ref{lem:two_bricks} and~\ref{lem:no_two_bricks}.
It follows from Lemma~\ref{lem:morse} that~$E_n$ is at
least~$(\eta(n)-1)$-connected.
\end{proof}

\begin{proposition}\label{prop:sublevel_conn}
For each~$n\ge1$, the pair~$(sX^{\le n},sX^{< n})$
is~$\eta(n)$-connected for~$s>1$, and the pair~$(1X^{\le n},1X^{<
n})$ is~$\nu(n)$-connected.
\end{proposition}

\begin{proof}
 Let~$x$ be a vertex in~$sX^{=n}$.  By
 Corollary~\ref{cor:desc_lk_conn}, the descending link $\dlk(x)$
 of~$x$ in~$sX$ is at least~$(\eta(n)-1)$-connected for~$s>1$, or
 $(\nu(n)-1)$-connected for~$s=1$.  The result now follows from
 Lemma~\ref{lem:morse}.
\end{proof}

We are now in a position to apply Brown's Criterion.

\begin{proof}[Proof of Main Theorem]
Consider the action of~$sV$ on~$sX$.  By
Corollary~\ref{cor:stein_space_cible}, $sX$ is contractible, by
Observation~\ref{obs:stabilizers}, the stabilizer of every cell is
finite\serialcomma and by Observation~\ref{obs:cocompactness}, each
$sX^{\le n}$ is finite modulo~$sV$.  By
Proposition~\ref{prop:sublevel_conn}, the connectivity of the
pairs~$(sX^{\le n},sX^{<n})$ tends to~$\infty$ as~$n$ tends
to~$\infty$.  Hence,~$sV$ is of type~$\F_\infty$ by Brown's
Criterion.
\end{proof}

% -------------------------------------------------------------------

%\bibliographystyle{alpha}
%\bibliography{bibdata}

\newcommand{\etalchar}[1]{$^{#1}$}

\end{document}